\documentclass[final,12pt]{elsarticle}
\usepackage{xcolor,soul,cancel}
\usepackage[color]{showkeys}
\definecolor{refkey}{rgb}{0,1,1}
\definecolor{labelkey}{rgb}{1,0,0}
\usepackage{amssymb}
\usepackage{amsmath}
\usepackage{amsthm}
\usepackage{graphicx}
\usepackage{float}
\journal{arXiv}

\newtheorem{thm}{Theorem}
\newtheorem{lem}{Lemma}
\newtheorem{cor}{Corollary}

\newcommand{\eq} [1] {\begin{equation}\label{#1}\quad}
\newcommand{\en} {\end{equation}}
\newcommand{\scal}[1]{\langle#1\rangle}
\newcommand{\norm}[1]{\left\Vert#1\right\Vert}
\newcommand{\abs}[1]{\left\vert#1\right\vert}

\newcommand{\C}{\mathbb C}

\newcommand{\conv}{\operatorname{conv}}

\newcommand{\cl}{\operatorname{cl}}

\begin{document}

\begin{frontmatter}
%\ead{ims2@nyu.edu, ilya@math.wm.edu, imspitkovsky@gmail.com}

\title{A note on the maximal numerical range.\tnoteref{support}}
%----------Author 1

%\ead{ims2@nyu.edu, ilya@math.wm.edu, imspitkovsky@gmail.com}

\author{Ilya M. Spitkovsky}
%\ead{ims2@nyu.edu, ilya@math.wm.edu, imspitkovsky@gmail.com}
\address{Division of Science and Mathematics, New York  University Abu Dhabi (NYUAD), Saadiyat Island,
P.O. Box 129188 Abu Dhabi, United Arab Emirates}

\tnotetext[support]{Supported in part by Faculty Research funding from the Division of Science and Mathematics, New York University Abu Dhabi.
\\ \hspace*{.5cm} Email addresses: ims2@nyu.edu, ilya@math.wm.edu, imspitkovsky@gmail.com}
%\tnotetext[email]{Email addresses: %ims2@nyu.edu, ilya@math.wm.edu, imspitkovsky@gmail.com}

\begin{abstract}
We show that the maximal numerical range of an operator has a non-empty intersection with the boundary of its numerical range if and only
if the operator is normaloid. A description of this intersection is also given.
\end{abstract}

\end{frontmatter}

First, let us set some notation and terminology.

For a subset $X$ of the complex plane $\C$, by $\cl X$, $\partial X$, and $\conv X$ we will denote the closure, boundary, and the convex hull of $X$, respectively.
\iffalse Given a topological space $X$, for any $Y\subset X$ we will denote by $\cl Y$ and $\partial Y$ the closure and the boundary of $Y$. Most of the time the role of $X$ will be played by the complex field $\C$.\fi

By an ``operator'' we throughout the paper understand a bounded linear operator acting on a Hilbert space $\mathcal H$.
The {\em numerical range} of such an operator $A$ is defined by the formula
\[ W(A)=\{ \scal{Ax,x}\colon x\in{\mathcal H},\ \norm{x}=1\}, \]
where $\scal{.,.}$ and $\norm{.}$ stand, respectively, for the scalar product on $\mathcal H$
and the norm associated with it. Introduced a century ago in the works by Toeplitz \cite{Toe18}
and Hausdorrf \cite{Hau} (and thus also known as the {\em Toeplitz-Hausdorff set}), it since has been a subject
of intensive research. We mention here only \cite{GusRa} as a standard source of references, and note the following basic
properties:

Due to the Cauchy-Schwarz inequality, the set $W(A)$ is bounded. Namely,
\eq{w} w(A):=\sup\{ \abs{z}\colon z\in W(A)\}\leq\norm{A}; \en
$w(A)$ is called the {\em numerical radius} of $A$.\iffalse , and the operator $A$ is {\em normaloid} if inequality \eqref{w}
turns into an equality.\fi

The set $W(A)$ is convex (the Toeplitz-Hausdorff theorem), and if $\dim\mathcal H<\infty$ it is also closed.

A (relatively) more recent notion of the {\em maximal} numerical range $W_0(A)$ was introduced in \cite{Sta70} as the set of
all $\lambda\in\C$ for which there exist \eq{mw} x_n\in\mathcal H \text{ such that } \norm{x_n}=1, \ \norm{Ax_n}\to\norm{A},
\text{ and } \scal{Ax_n,x_n}\to\lambda.\en

It was also shown in \cite[Lemma 2]{Sta70} that $W_0(A)$ is convex, closed, and is contained in the closure of $W(A)$:
\eq{cont} W_0(A)\subset \cl W(A). \en
Observe that in the finite dimensional case $W_0(A)=W(B)$, where $B$ is the compression of $A$
onto the eigenspace of $A^*A$ corresponding to its maximal eigenvalue, so the above mentioned properties of the
maximal numerical range are rather straightforward.

Given the inclusion \eqref{cont}, it is natural to try to describe in more detail the positioning of $W_0(A)$
with respect to $W(A)$. In particular, what can be said about the points of $W_0(A)$ which lie on the boundary $\partial W(A)$ of $W(A)$?
%In this short note we characterize the operators $A$ for which \eq{b} W_0(A)\cap\partial W(A)\neq\emptyset.\en

We start by describing the intersection of $W_0(A)$ with the circle \[ {\cal C}_A:=\{ z\colon \abs{z}=\norm{A}\}.\]
\begin{lem}\label{l:nor}For any operator $A$,
\eq{int} W_0(A)\cap {\mathcal C}_A= \cl W(A)\cap {\mathcal C}_A= \sigma(A)\cap {\mathcal C}_A.\en
\end{lem}
\begin{proof} The second equality in \eqref{int} is well known \cite[Problem 218]{Hal82}. Due to \eqref{cont} it therefore remains to prove only
the inclusion of the middle term in the left hand side. To this end, observe that with any $\lambda\in \cl W(A)$ by definition
there is associated a sequence of unit vectors $x_n\in\mathcal H$ for which $\scal{Ax_n,x_n}\to\lambda$. If, in addition,
$\abs{\lambda}=\norm{A}$, then the Cauchy-Schwarz inequality implies that $\norm{Ax_n}\to\norm{A}$. In other words,
\eqref{mw} holds. \end{proof}
Recall that operators for which the second (and thus, equivalently, the third) term in \eqref{int} is non-empty are called
{\em normaloid}. Therefore, Lemma~\ref{l:nor} implies the sufficiency in the following

\begin{thm}\label{th:nor}The intersection $W_0(A)\cap\partial W(A)$ is non-empty if and only the operator $A$ is normaloid. \end{thm}
{\em Proof of necessity.} To simplify the notation, without loss of generality suppose that $\norm{A}=1$; this can be arranged by an
appropriate scaling not having any effect on the validity of the statement. Then ${\mathcal C}_A$ is simply the unit circle $\mathbb T$.

If $W_0(A)\cap\partial W(A)\neq\emptyset$, then \eqref{mw} holds for some $\lambda\in\partial W(A)$. Choose a unit vectors $y_n$
orthogonal to $x_n$ and such that $Ax_n$ lies in the span $\mathcal L_n$ of $x_n$ and $y_n$. Then \[ Ax_n=a_nx_n+c_ny_n \] for some $a_n,c_n\in\C$,
\eq{ac} a_n\to\lambda, \text{ and} \abs{c_n}^2\to 1-\abs{\lambda}^2. \en Consider the compression of $A$ onto $\mathcal L_n$. Its matrix $A_n$
with respect to the basis $\{ x_n,y_n\}$ has $[a_n,c_n]^T$ as its first column; denote the second column of $A_n$ as $[b_n,d_n]^T$.

Passing to a subsequence if needed, we may suppose that \[ A_n\to A_0:=\begin{bmatrix} a & b\\ c & d\end{bmatrix}, \] where due to
\eqref{ac} $a=\lambda$ and $\abs{c}^2=1-\abs{\lambda}^2$.

Since $W(A_n)\subset W(A)$ for all $n=1,2,\ldots$, we have $W(A_0)\subset\cl W(A)$, and so the $(1,1)$-entry $a$ of $A_0$ lies on the boundary
of its numerical range. This is only possible if $\abs{b}=\abs{c}$, as was observed e.g. in \cite[Corollary 4]{ZDWD}, see also
\cite[Proposition 4.3]{GauWu13}. Moreover,
\eq{aa} A_0^*A_0=\begin{bmatrix} 1 & \overline{a}b+\overline{c}d \\ \overline{b}a+\overline{d}c & \abs{b}^2+\abs{d}^2\end{bmatrix}. \en
Since $\norm{A_0}=\lim\norm{A_n}\leq 1$, the matrix  $A_0^*A_0$ must be diagonal. When combined with the already established equality
$\abs{b}=\abs{c}$, this implies that either $b=c=0$, or $\abs{d}=\abs{a}$.

In the former case the normaloidness of $A$ is immediate, because then $\lambda\in\mathbb T$. In the latter case \eqref{aa} simplifies to
$A_0^*A_0=I$, i.e. $A_0$ is unitary. Then $\cl W(A)\cap\mathbb T\supset\sigma(A_0)\neq\emptyset$, also implying that $A$ is normaloid.
\qed

It follows from Theorem~\ref{th:nor} in particular that an operator $A$ is normaloid if and only if its numerical radius $w(A)$ coincides
with $w_0(A):=\max\{ \abs{z} \colon z\in W_0(A)\}$. This result was established in \cite[Corollary 1]{ChCh}. Moreover, the paper \cite{ChCh}
served as a motivation for the present note, and our proof of Theorem~\ref{th:nor} is making use of some reasoning from \cite{ChCh}.

A closer look at the proof of Theorem~\ref{th:nor} yields an explicit description of the set $W_0(A)\cap\partial W(A)$.
\begin{cor}\label{co:int}The intersection of $W_0(A)$ with the boundary of $W(A)$ consists of  $\sigma(A)\cap{\mathcal C}_A$ and the set ${\mathcal L}_A$ of all the chords of ${\mathcal C}_A$ lying on $\partial W(A)$:
\[ W_0(A)\cap\partial W(A)=\left(\sigma(A)\cap{\mathcal C}_A\right)\cup {\mathcal L}_A.\]
\end{cor}

Note that the endpoints of the above mentioned chords belong to $\sigma(A)\cap{\mathcal C}_A$. Considering by convention the remaining points of $\sigma(A)\cap{\mathcal C}_A$ as the endpoints of ``degenerate'' zero-length chords
of ${\mathcal C}_A$, we may say simply that $W_0(A)\cap\partial W(A)$ is exactly the set of all
chords of ${\mathcal C}_A$ lying on $\partial W(A)$.

Being convex, along with $\sigma(A)\cap{\mathcal C}_A$ the set $W_0(A)$ must also contain its convex hull
$\conv(\sigma(A)\cap{\mathcal C}_A)$. Since ${\mathcal L}_A\subset \conv(\sigma(A)\cap{\mathcal C}_A)$,
the equality \eq{ch} W_0(A)=\conv(\sigma(A)\cap{\mathcal C}_A)\en is plausible. It may fail, however, even in finite dimensions.

{\sl Example.} Let \[ B =\begin{bmatrix} 0 & 1 & 0 \\ 0 & 0 & 0 \\ 0 & 0 & 1 \end{bmatrix}. \]
Then $\norm{B}=1$ is attained on the 2-dimensional span of the standard basis vectors $e_2,e_3$. The compression of $A$ onto their span is
the matrix $\begin{bmatrix} 0 & 0 \\ 0 & 1\end{bmatrix}$, and so $W_0(B)$ is the line segment $[0,1]$. On the other hand, $W(A)$ is the ice-cone shaped convex hull of the circle centered at the origin and having radius $1/2$ and the point $1$. In particular, $w(B)=1$, making $A$ normaloid.
In agreement with Corollary~\ref{co:int} we have (taking into consideration that ${\mathcal C}_B=\mathbb T$):
\[ W_0(B)\cap\partial W(B)=\sigma(B)\cap{\mathbb T}=\{1\}, \]
and so $\conv(\sigma(B)\cap{\mathbb T})=\{1\}$ is a proper subset of $W_0(B)$.

The situation changes if $A$ is normal and not merely normaloid.

\begin{thm}\label{th:nor1}Equality \eqref{ch} holds for normal operators $A$.\end{thm}
\begin{proof} Due to the inclusion $W_0(A)\supset \sigma(A)\cap{\mathcal C}_A$ and the fact that both sides in \eqref{ch} are convex and
compact, it suffices to show that any open half-plane containing $\sigma(A)\cap{\mathcal C}_A$ also contains $W_0(A)$.

So, consider a half-plane $\Pi\supset\sigma(A)\cap{\mathcal C}_A$. The spectrum of $A$ is disjoint with the arc $\gamma={\mathcal C}_A\setminus\Pi$,
and the distance between $\gamma$ and $\sigma(A)$ is therefore positive. Denoting it by $\epsilon$, observe that
\[ \sigma_\epsilon(A):=\{\lambda\in\sigma(A)\colon \abs{\lambda}\geq \norm{A}-\epsilon\}\subset\Pi. \]
Let $A_\epsilon$ be the restriction of $A$ onto its spectral subset corresponding to $\sigma_\epsilon(A)$. The definition of $W_0(A)$ implies that
$W_0(A)\subset\cl W(A_\epsilon)$. On the other hand, the operator $A_\epsilon$ is normal along with $A$, and so
$\cl W(A_\epsilon)=\conv\sigma(A_\epsilon)=\conv\sigma_\epsilon(A)\subset\Pi$.
\end{proof}

Recall that an operator $A$ acting on a Hilbert space $\mathcal H$ is {\em subnormal} if there exists a Hilbert space $\mathcal G$ and operators
$B\colon {\mathcal G}\to {\mathcal H}$, $C\colon {\mathcal G}\to {\mathcal G}$ such that the operator
\[ N:=\begin{bmatrix} A & B \\ 0 & C\end{bmatrix}\colon {\mathcal H}\oplus{\mathcal G}\to {\mathcal H}\oplus{\mathcal G}\]
is normal. As it happens, property \eqref{ch} extends from normal to subnormal operators.
\begin{cor}\label{co:sub}Equality \eqref{ch} holds for subnormal operators $A$.\end{cor}
\begin{proof}
Consider the minimal normal extension $N$ of $A$, the existence and properties of which are discussed e.g. in \cite{Conw91,Hal82}. It is true
in particular that $\norm{A}=\norm{N}$. So, whenever a sequence of unit vectors $x_n\in\mathcal H$ is such that $\norm{Ax_n}\to\norm{A}$,
it at the same time satisfies $\norm{Nx_n}\to\norm{N}$. Consequently, \eq{wan} W_0(A)\subset W_0(N).\en

Furthermore, $\sigma(A)$ equals $\sigma(N)$ with some holes filled, and so \eq{san} \sigma(A)\cap{\mathcal C}_A=\sigma(N)\cap{\mathcal C}_N.\en
Combining \eqref{wan},\eqref{san} with the equality $W_0(N)=\conv(\sigma(N)\cap{\mathcal C}_N)$ which holds due to Theorem~\ref{th:nor1},
we obtain
\[ W_0(A)\subset W_0(N)=\conv(\sigma(N)\cap{\mathcal C}_N)=\conv(\sigma(A)\cap{\mathcal C}_A). \]
Since the converse inclusion holds for any $A$, we are done.
\end{proof}

\providecommand{\bysame}{\leavevmode\hbox to3em{\hrulefill}\thinspace}
\providecommand{\MR}{\relax\ifhmode\unskip\space\fi MR }
\providecommand{\MRhref}[2]{%
  \href{http://www.ams.org/mathscinet-getitem?mr=#1}{#2}
}
\providecommand{\href}[2]{#2}

\iffalse
\bibliographystyle{amsplain}
\bibliography{master}
\fi
\end{document}